\documentclass[a4paper]{article}
\usepackage{fullpage}
\usepackage{amsmath}
\usepackage{amssymb}
\usepackage{amsthm}
\usepackage{graphicx}
\usepackage{bbm}
\usepackage{enumerate}
\usepackage{microtype}
\usepackage{tocbibind}
\usepackage{xcolor}
\usepackage{indentfirst}
\usepackage[none]{hyphenat}
\usepackage[colorlinks=true,urlcolor=blue,linkcolor=blue,plainpages=false,pdfpagelabels]{hyperref}
\allowdisplaybreaks

\newtheorem{theorem}{Theorem}[section]
\newtheorem{proposition}[theorem]{Proposition}
\newtheorem{lemma}[theorem]{Lemma}

\newtheorem{definition}[theorem]{Definition}

\newcommand{\R}{\mathbb{R}}

\newcommand{\eps}{\varepsilon}

\newcommand{\se}{\subseteq}

\DeclareMathOperator{\argmin}{argmin}
\DeclareMathOperator{\Prox}{Prox}
\title{Further quantitative moduli around uniform convexity}
\author{Andrei Sipo\c s\\[2mm]
\footnotesize Research Center for Logic, Optimization and Security (LOS), Department of Computer Science,\\
\footnotesize Faculty of Mathematics and Computer Science, University of Bucharest,\\
\footnotesize Academiei 14, 010014 Bucharest, Romania\\[1mm]
\footnotesize Simion Stoilow Institute of Mathematics of the Romanian Academy,\\
\footnotesize Calea Grivi\c tei 21, 010702 Bucharest, Romania\\[2mm]
\footnotesize Email: andrei.sipos@fmi.unibuc.ro\\
}
\date{}

\begin{document}

\maketitle

\begin{abstract}
We derive new formulas for quantities involved in the study of uniformly convex spaces, linear and nonlinear.

\noindent {\em Mathematics Subject Classification 2020}: 51F99, 52A55, 53C23.

\noindent {\em Keywords:} Uniform convexity, hyperbolic spaces, projection, proximal mapping.
\end{abstract}

\section{Introduction}

Proof mining is a research paradigm (first suggested by Georg Kreisel in the 1950s under the name of `unwinding of proofs’ and then highly developed starting in the 1990s by Ulrich Kohlenbach and his students and collaborators), which aims to apply tools from proof theory (a branch of mathematical logic) to proofs in mainstream mathematics in order to uncover additional -- typically quantitative -- information which may not be readily apparent. In the past 25 years, proof mining took off by obtaining quantitative results in diverse areas like nonlinear analysis, ergodic theory, approximation theory, commutative algebra, probability theory, for more information see Kohlenbach's 2008 monograph \cite{Koh08} or his more recent ICM survey \cite{Koh19}.

This paper focuses on two illustrative case studies from nonlinear analysis where this methodology may be applied in order to obtain new quantitative moduli.

In 1981, J. Pr\"u{\ss} \cite{Pru81} proved a characterization of uniform convexity in Banach spaces, which seems to be an analogue of the uniform continuity on bounded subsets of the duality mapping in uniformly smooth spaces, for whose quantitative treatment see \cite{KohLeu12}, and for the exact connection between the two properties, see \cite[Chapters 4 and 5]{Chi09}. The advantage of Pr\"u{\ss}'s proof is that it is quite short and highly direct (e.g. the definition of uniform convexity is applied as it is), yet slightly logically non-trivial, and thus constitutes an appropriate test case for proof mining. More concretely, the proof uses tools which are second nature to analysts, but involving foundationally strong logical principles, like forming sequences, passing to limits superior and inferior and using various inequalities regarding such limits. A naive analysis of the proof would, of course, involve, alternating quantifiers for these limits and corresponding comprehension principles. What we show is that the proof may be massively logically simplified via replacing sequences by fixed points, making intermediate statements more `approximate', at the expense of the argument becoming more messy, finally obtaining a proof from which the corresponding modulus may be easily read. This is what we do in Section~\ref{sec:prus}.

In the second case study, described in Section~\ref{sec:prox}, we move on to a class of `nonlinear' spaces -- metric spaces which are endowed with a generalization of the usual linear convexity structure -- where we treat a property of the proximal mapping, a tool recently extended to this setting in \cite{SipXX}. Here, the main difficulty involves the elimination of tools used in uniformly Banach spaces which are specific to that setting: the dual and the subdifferential. We show how one may replace those my simple metric arguments and thus derive the sought modulus.

\section{Preliminaries}\label{sec:prelim}

\subsection{Linear spaces}

All Banach spaces in this paper will contain a non-zero element.

Uniform convexity in Banach spaces was first introduced by Clarkson \cite{Cla36}. We use the following formulation: a Banach space $X$ is {\it uniformly convex} if there is an $\eta : (0,\infty) \to (0,1]$, called a {\it modulus of uniform convexity}, such that for all $\eps>0$ and all $x,y \in X$ with $\|x\|\leq 1$, $\|y\| \leq 1$ and $\|x-y\| \geq \eps$ one has that
$$\left\|\frac{x+y}2\right\| \leq 1-\eta(\eps).$$
One typically defines the fixed modulus of uniform convexity $\delta_X$ by putting, for all $\eps \in (0,2]$,
$$\delta_X(\eps):=  \inf \left\{ 1- \left\| \frac{x+y}2 \right\| \bigm| \|x\|\leq 1, \|y\|\leq 1, \|x-y\| \geq \eps \right\}$$
and says that $X$ is uniformly convex if and only if, for all $\eps \in (0,2]$, $\delta_X(\eps)>0$. It is immediate that the two definitions coincide: in this case, $\delta_X$ may serve as a modulus $\eta$ in our sense (it is, in fact, the largest such modulus) by putting, for all $\eps>0$, $\eta(\eps):=\delta_X(\min(\eps,2))$.

It is immediate that the use of this modulus scales, as follows.

\begin{proposition}\label{mod-r}
Let $X$ be a uniformly convex Banach space with modulus $\eta : (0,\infty) \to (0,1]$. Then, for all $\eps$, $r>0$ and all $x,y \in X$ with $\|x\|\leq r$, $\|y\| \leq r$ and $\|x-y\| \geq \eps r$ one has that
$$\left\|\frac{x+y}2\right\| \leq (1-\eta(\eps))r.$$
\end{proposition}

\begin{definition}
Let $X$ be a Banach space. We define the {\bf normalized duality mapping of $X$} to be the map $J :X\to 2^{X^*}$, defined, for all $x \in X$, by
$$ J(x) := \{x^*\in X^* \mid  x^*(x)=\|x\|^2,\ \|x^*\|=\|x\| \}.$$
\end{definition}

It is known (see \cite[Lemma 3.4]{Chi09}) that, for any space $X$ and any $x \in X$, we have that $J(x)\neq\emptyset$. We shall, occasionally, but not always, denote, for all spaces $X$, all $x^* \in X^*$ and $y \in X$, $x^*(y)$ by $\langle y,x^* \rangle$ (by analogy to Hilbert spaces). 

\subsection{Nonlinear spaces}

One says that a metric space $(X,d)$ is {\it geodesic} if for any two points $x$, $y \in X$ there is a {\it geodesic} that joins them, i.e.\ a mapping $\gamma : [0,1] \to X$ such that $\gamma(0)=x$, $\gamma(1)=y$ and for any $t$, $t' \in [0,1]$ we have that
$$d(\gamma(t),\gamma(t')) = |t-t'| d(x,y).$$

Generally, geodesics need not be unique, and, thus, one is led to the following definition, which considers geodesics on a metric space as an additional structure\footnote{The convexity function $W$ was first considered by Takahashi in \cite{Tak70} where a triple $(X,d,W)$ satisfying $(W1)$ is called a convex metric space. The notion used here, frequently considered nowadays to be the nonlinear generalization of convexity in normed spaces, was introduced by Kohlenbach in \cite{Koh05}. As said in the Introduction, see \cite[pp. 384--388]{Koh08} for a detailed discussion on the relationship between various definitions of hyperbolicity and on the proof-theoretical considerations that ultimately led to the adoption of this one. We only mention here that this notion is more general than that of hyperbolic spaces in the sense of Reich and Shafrir \cite{ReiSha90}, and slightly more restrictive than the setting due to Goebel and Kirk \cite{GoeKir83} of spaces of hyperbolic type. } (and not as a property), which is required to satisfy additional properties.

\begin{definition}
A {\bf $W$-hyperbolic space} is a triple $(X,d,W)$ where $(X,d)$ is a metric space and $W: X^2 \times [0,1] \to X$ such that, for all $x$, $y$, $z$, $w \in X$ and $\lambda$, $\mu \in [0,1]$, we have that
\begin{enumerate}[(W1)]
\item $d(z,W(x,y,\lambda)) \leq (1-\lambda)d(z,x) + \lambda d(z,y)$;
\item $d(W(x,y,\lambda),W(x,y,\mu)) = |\lambda-\mu|d(x,y)$;
\item $W(x,y,\lambda)=W(y,x,1-\lambda)$;
\item $d(W(x,z,\lambda),W(y,w,\lambda))\leq(1-\lambda) d(x,y) + \lambda d(z,w)$.
\end{enumerate}
\end{definition}

Clearly, any normed space may be made into a $W$-hyperbolic space in a canonical way. 

A subset $C$ of a $W$-hyperbolic space $(X,d,W)$ is called {\it convex} if, for any $x$, $y \in C$ and $\lambda \in [0,1]$, $W(x,y,\lambda) \in C$. If $(X,d,W$) is a $W$-hyperbolic space, $x$, $y \in X$ and $\lambda \in [0,1]$, we denote the point $W(x,y,\lambda)$ by $(1-\lambda)x+\lambda y$. We will, also, mainly write $\frac{x+y}2$ for $\frac12 x + \frac12 y$.

\begin{proposition}\label{quad}
Let $(X,d,W)$ be a $W$-hyperbolic space, $\lambda \in [0,1]$ and $x$, $y$, $a \in X$. Then
$$d^2\left((1-\lambda)x+\lambda y,a\right) \leq (1-\lambda)d^2(x,a)+\lambda d^2(y,a).$$
\end{proposition}

\begin{proof}
Using $(W1)$, we get that
$$d^2\left((1-\lambda)x+\lambda y,a\right) \leq \left((1-\lambda)d(x,a)+\lambda d(y,a)\right)^2 \leq (1-\lambda)d^2(x,a)+\lambda d^2(y,a),$$
the last inequality following from the convexity of the square function.
\end{proof}

As per \cite{Koh05,Leu07}, a particular nonlinear class of $W$-hyperbolic spaces is the one of CAT(0) spaces, introduced by A. Aleksandrov \cite{Ale51} and named as such by M. Gromov \cite{Gro87}, which may be defined, by the discussion in \cite[pp. 387--388]{Koh08}, as those $W$-hyperbolic spaces $(X,d,W)$ such that, for any $a$, $x$, $y \in X$,
\begin{equation}\label{cat0}
d^2\left(\frac{x+y}2,a\right) \leq \frac12 d^2(x,a) + \frac12 d^2(y,a) - \frac14 d^2(x,y).
\end{equation}

In particular, any inner product space is a CAT(0) space, and, indeed, as experience shows, CAT(0) spaces may be regarded as the rightful nonlinear generalization of inner product spaces.

Uniform convexity was first introduced by Leu\c stean \cite{Leu07,Leu10}  in this hyperbolic setting, as motivated by \cite[p. 105]{GoeRei84}. The monotonicity condition was justified by proof-theoretical considerations.

\begin{definition}
If $(X,d,W)$ is a $W$-hyperbolic space, then a {\bf modulus of uniform convexity} for $(X,d,W)$ is a function $\eta :(0, \infty) \times (0,2] \to (0,1]$ such that, for any $r >0$, any $\eps \in (0,2]$ and any $a$, $x$, $y \in X$ with $d(x,a) \leq r$, $d(y,a) \leq r$, $d(x,y) \geq \eps r$, we have that
$$d\left(\frac{x+y}2,a\right) \leq (1-\eta(r,\eps))r.$$
We call the modulus {\bf monotone} if, for any $r$, $s >0$ with $s \leq r$ and any $\eps \in (0,2]$, we have that $\eta(r,\eps) \leq \eta(s,\eps)$.

A {\bf $UCW$-hyperbolic space} is a $W$-hyperbolic space that admits a monotone modulus of uniform convexity.
\end{definition}

As remarked in \cite[Proposition 2.6]{Leu07}, CAT(0) spaces are $UCW$-hyperbolic spaces having as a modulus of uniform convexity the simple function $(r,\eps) \mapsto \eps^2/8$. Recently, together with Pedro Pinto, we produced, for the first time, a non-normed {\bf and} non-CAT(0) example of a $UCW$-hyperbolic space \cite{PinSip}.

In a metric space $(X,d)$, a nonempty subset $C$ of $X$ is a {\it Chebyshev set} if, for any $x \in X$, there exists a unique $y \in C$ such that, for any $z \in C$, $d(x,y) \leq d(x,z)$ -- this defines the {\it projection mapping} $P_C:X \to C$. By \cite[Proposition 2.4]{Leu10}, any closed, convex, nonempty subset of a complete $UCW$-hyperbolic space is a Chebyshev set.

We will now discuss a generalization of the projection. The following notions are standard in convex optimization.
\begin{definition}
Let $(X,d,W)$ be a complete $UCW$-hyperbolic space and $f : X \to \R \cup \{\infty\}$. We say that $f$ is:
\begin{itemize}
\item {\bf proper} if there is an $x \in X$ with $f(x) \neq \infty$;
\item {\bf convex} if, for all $x$, $y \in X$ and all $\lambda \in [0,1]$, $f((1-\lambda)x+\lambda y) \leq (1-\lambda)f(x) + \lambda f(y)$;
\item {\bf lower semicontinuous (lsc)} if, for all $x \in X$, $r \in \R$ with $r<f(x)$ there is a $\delta >0$ such that, for all $z \in X$ with $d(x,z) \leq \delta$, $r<f(z)$.
\end{itemize}
\end{definition}

In another paper \cite{SipXX}, we showed that one may extend proximal minimization, an established tool of convex optimization, to this uniformly convex hyperbolic setting (the first such nonlinear generalization, namely to complete CAT(0) spaces, was introduced by Jost \cite{Jos95}). Concretely, we may define, for any complete $UCW$-hyperbolic space $(X,d,W)$ and for any proper, convex, lsc function $f : X \to \R \cup \{\infty\}$, the {\it proximal mapping of $f$}, $\Prox_f : X \to X$, by putting, for any $x \in X$,
$$\Prox_fx := \argmin_{y \in X} \left(f(y) + \frac12 d^2(x,y)\right).$$
This operator generalizes the projection onto a closed, convex, nonempty subset $C$, since the latter is just the proximal mapping of the indicator function $\iota_C$ (which is $0$ on $C$ and $\infty$ otherwise).

\section{A property equivalent to uniform convexity}\label{sec:prus}

The following result was proven by Pr\"u{\ss} in \cite{Pru81}.

\begin{theorem}[{\cite[Theorem 1]{Pru81}}]\label{pru} 
Let $X$ be a Banach space. We have that $X$ is uniformly convex iff for any $R >0$ there is a function $\omega_R : [0, \infty) \to [0,\infty)$, which is non-decreasing and $0$ is its only zero, such that, for any $x$, $y \in X$ with $\|x\|$, $\|y\| \leq R$ and for any $x^* \in Jx$ and $y^* \in Jy$, we have that
$$\langle x-y, x^*-y^* \rangle \geq \omega_R(\|x-y\|) \cdot \|x-y\|.$$
\end{theorem}

Our goal is to obtain a quantitative version of this theorem, i.e. a computable modulus. In order to cut down on the complexity of the result, we apply an oft-repeated piece of proof mining wisdom (due to U. Kohlenbach) that `one should replace the monotonicity of the function by the monotonicity of the statement'. This is encapsulated in the following.

\begin{proposition}
Let $X$ be a Banach space. The condition in Theorem~\ref{pru} is equivalent to the following: there is an $\omega : (0,\infty) \times (0, \infty) \to (0, \infty)$ such that, for any $R$, $\rho >0$, for any $x$, $y \in X$ with $\|x-y\| \geq \rho$ and $\|x\|$, $\|y\| \leq R$ and for any $x^* \in Jx$ and $y^* \in Jy$, we have that
$$\langle x-y, x^*-y^* \rangle \geq \omega(R,\rho) \cdot \|x-y\|.$$
\end{proposition}

\begin{proof}
From left to right, we just take, for any $R$, $\rho >0$, $\omega(R,\rho):=\omega_R(\rho)$, and the result follows from the monotonicity of each $\omega_R$.

From right to left, fix $R>0$. We define (as in Pr\"u{\ss}'s proof): $\omega_R(0):=0$; for any $\rho \in (0,2R]$,
$$\omega_R(\rho) := \inf \left\{ \frac{\langle x-y, x^*-y^* \rangle}{\|x-y\|} \mid \text{$x$, $y \in X$, $\|x\|$, $\|y\| \leq R$, $\|x-y\| \geq \rho$, $x^* \in Jx$, $y^* \in Jy$}\right\};$$
and, for any $\rho > 2R$, $\omega_R(\rho):=\omega_R(2R)$.

For a given $\rho \in (0,2R]$, the set of which the infimum is taken is non-empty because one may take an $x \in X$ with $\|x\|=R$ and $y:=-x$, so $\|x-y\|=2R \geq \rho$, and the infimum is non-negative because, for any $x$, $y$ as there,
$$\langle x-y, x^*-y^* \rangle = \|x\|^2 + \|y\|^2 - x^*(y) - y^*(x) \geq \|x\|^2 + \|y\|^2 - 2\|x\|\|y\| = (\|x\|-\|y\|)^2\geq 0.$$

The fact that $\omega_R$ is non-decreasing is immediate.

It remains to be shown that, for any $\rho \in (0,2R]$, $\omega_R(\rho)>0$. Take such a $\rho$ and take $x$, $y \in X$ with $\|x\|$, $\|y\| \leq R$ and $\|x-y\| \geq \rho$, and take $x^* \in Jx$, $y^* \in Jy$, so, by our assumption,
$$\frac{\langle x-y, x^*-y^* \rangle}{\|x-y\|} \geq \omega(R,\rho).$$
Thus, we have that $\omega_R(\rho)\geq \omega(R,\rho)>0$.
\end{proof}

We may simplify the statement even more.

\begin{proposition}
Let $X$ be a Banach space. The condition in Theorem~\ref{pru} is further equivalent to the following: there is an $\alpha : (0,\infty) \times (0, \infty) \to (0, \infty)$ such that, for any $R$, $\rho >0$, for any $x$, $y \in X$ with $\|x-y\| \geq \rho$ and $\|x\|$, $\|y\| \leq R$ and for any $x^* \in Jx$ and $y^* \in Jy$, we have that
$$\langle x-y, x^*-y^* \rangle \geq \alpha(R,\rho).$$
\end{proposition}

\begin{proof}
It is immediate that one may take, in one direction, $\alpha(R,\rho):=\omega(R,\rho) \cdot \rho$, and, in the other, $\omega(R,\rho):=\alpha(R,\rho)/(2R)$.
\end{proof}

It is for this modulus above that we shall find a quantitative expression.

\begin{theorem}
Define, for any $\eta : (0,\infty) \to (0,1]$ and any $R$, $\rho > 0$,
$$\alpha_\eta(R, \rho) := \min\left( \frac{\rho^2}{64}\eta^2\left(\frac\rho R\right),\ \frac{\rho^2}8\eta\left(\frac\rho R\right) \right).$$
Let $X$ be a uniformly convex Banach space with modulus $\eta : (0,\infty) \to (0,1]$. Let $R$, $\rho >0$, let $x$, $y \in X$ be with $\|x-y\| \geq \rho$ and $\|x\|$, $\|y\| \leq R$ and let $x^* \in Jx$ and $y^* \in Jy$. Then we have that
$$\langle x-y, x^*-y^* \rangle \geq \alpha_\eta(R,\rho).$$
\end{theorem}

\begin{proof}
We shall write $\alpha$ for $\alpha_\eta(R,\rho)$. Put also $c:=\frac14\eta\left(\frac\rho R\right)$, so
$$\alpha = \min\left(\frac{\rho^2}{16} \cdot c^2,\ \frac{\rho^2}2 \cdot c\right).$$ 

Assume towards a contradiction that $\langle x-y, x^*-y^* \rangle < \alpha$. Note that we also have that
$$(\|x\|-\|y\|)^2 = \|x\|^2 + \|y\|^2 - 2\|x\|\|y\| \leq   \|x\|^2 + \|y\|^2 - x^*(y) - y^*(x) = \langle x-y, x^*-y^* \rangle < \alpha.$$

Assume w.l.o.g. that $\|x\|\leq \|y\|$, and, since $\|x-y\| \geq \rho > 0$, we have that $\|y\|>0$. Put $a:=\|y\| \leq R$. Note that we have that $a \geq \rho/2$, since, otherwise, we would have that
$$\|x-y\| \leq \|x\| + \|y\| \leq 2a < \rho,$$
a contradiction.

We have that
$$\alpha \leq \frac{\rho^2}{16} \cdot c^2 \leq \frac{a^2c^2}4,$$
so
$$1-\frac{\sqrt{\alpha}}a \geq 1-\frac c 2 \geq \sqrt{1-c},$$
i.e. that $a- \sqrt{\alpha} \geq a\sqrt{1-c}$. On the other hand, since $(\|x\|-\|y\|)^2 < \alpha$, and $\|x\|\leq \|y\|=a$, we have that $a-\|x\| < \sqrt{\alpha}$, from which we get that $\|x\|^2 > (a-\sqrt{\alpha})^2 \geq a^2(1-c)$.

We also have that
$$\alpha \leq \frac{\rho^2}2 \cdot c \leq 2a^2c.$$

Since $\|x-y\| \geq \rho \geq (\rho/R) \cdot a$, we have, by Proposition~\ref{mod-r}, that
$$\left\|\frac{x+y}2\right\| \leq \left(1-\eta\left(\frac\rho R\right)\right) \cdot a,$$
so
$$\|x+y\| \leq 2a\left(1-\eta\left(\frac\rho R\right)\right),$$
and thus that
$$\langle x^*, x+y \rangle \leq \|x\| \cdot \|x+y\| \leq 2a^2\left(1-\eta\left(\frac\rho R\right)\right).$$

On the other hand, using the reasoning from \cite[Theorem 3.2]{Kie02}, we have that
\begin{align*}
\langle x^*, x+y \rangle &= y^*(x) + x^*(y) - (y^*(x) - x^*(x)) \\
&= \|x\|^2 + \|y\|^2 - \langle x-y, x^*-y^* \rangle - (y^*(x) - \|x\|^2) \\
&> 2\|x\|^2 + \|y\|^2 - \|x\|\|y\| - \alpha \\
&= 2\|x\|^2 + \|y\| (\|y\| - \|x\|)  - \alpha \\
&\geq 2\|x\|^2 - \alpha\\
&\geq 2a^2(1-c) - 2a^2c\\
&= 2a^2\left(1-\frac12\eta\left(\frac\rho R\right)\right),
\end{align*}
a contradiction.
\end{proof}

Note that the above result mentions the dual in its statement, so we cannot for now attempt to generalize it to $UCW$-hyperbolic spaces, where we do not have a workable notion of a dual. Our hope comes from finding theorems which only make use of the dual in their {\it proof}, so that use could be removed and the theorem generalized to the nonlinear setting. This is what we do in the next section.

\section{A property regarding proximal mappings in nonlinear spaces}\label{sec:prox}

We have the following (almost trivial) property of the projection.

\begin{lemma}
Let $B$, $r >0$. Let $(X,d,W)$ be a complete $UCW$-hyperbolic space and $C \se X$ be closed, convex, nonempty. Let $x$, $y \in X$ be such that $d(x,P_Cx) \leq B$ and $d(x,y) \leq r$.

Then $d(y,P_Cy) \leq r+B$.
\end{lemma}

\begin{proof}
We have that $d(y,P_Cy) \leq d(y,P_Cx) \leq d(y,x) + d(x,P_Cx) \leq r+B$.
\end{proof}

We would like to generalize the above to proximal mappings. In uniformly convex Banach spaces, this was shown by Ba\v{c}\'ak and Kohlenbach \cite[Lemma 3.9]{BacKoh18}, using the following result.

\begin{lemma}[{\cite[Lemma 3.8]{BacKoh18}}]
Let $(X,\|\cdot\|)$ be a uniformly convex Banach space and $f$ be a proper, convex, lsc function on $X$. Let $x \in X$.

Then there is a $z^* \in J(x-\Prox_f x)$ such that, for all $y \in X$,
$$f(\Prox_fx) - f(y) \leq \langle z^*, \Prox_fx -y\rangle.$$
\end{lemma}

The problem with generalizing this is that the statement uses the dual and the proof uses the subdifferential, and we do not have readily available generalizations of these in the nonlinear setting. On the other hand, one may notice that the relevant content here is that (using the variables in the statement)
$$f(\Prox_fx) - f(y) \leq \|x-\Prox_f x\| \cdot \|y-\Prox_fx\|.$$
This can be formulated in the hyperbolic setting. Moreover, it can be proven using, instead of the subdifferential, just the usual real derivative.

\begin{lemma}
Let $(X,d,W)$ be a complete $UCW$-hyperbolic space and $f$ be a proper, convex, lsc function on $X$. Let $x$, $y \in X$.

Then $f(\Prox_fx) - f(y) \leq d(x,\Prox_fx) \cdot d(y,\Prox_fx)$.
\end{lemma}

\begin{proof}
Put $z:=\Prox_f x$. Let $t \in (0,1)$ and put $v:=(1-t)z+ty$. By the definition of the proximal mapping and the convexity of $f$, we have that
$$f(z) + \frac12d^2(x,z) \leq f(v) + \frac12 d^2(x,v) \leq (1-t)f(z) + tf(y) + \frac12 d^2(x,v),$$
so
$$t(f(z)-f(y)) \leq \frac12(d^2(x,v) - d^2(x,z)).$$
On the other hand, we have that
$$d(x,v) \leq (1-t)d(x,z)+td(x,y) \leq (1-t)d(x,z) + td(x,z) + td(y,z) =  d(x,z) + td(y,z),$$
so
$$d^2(x,v) \leq d^2(x,z) + 2td(x,z)d(y,z) + t^2d^2(y,z).$$
By putting these together, we get that
$$t(f(z)-f(y)) \leq \frac12(2td(x,z)d(y,z) + t^2d^2(y,z)),$$
so
$$f(z) - f(y) \leq d(x,z) \cdot d(y,z) + \frac t 2 d^2(y,z),$$
and, by taking $t \to 0$, we get the conclusion.
\end{proof}

We may now prove the desired generalization.

\begin{lemma}
Let $B$, $r >0$. Let $(X,d,W)$ be a complete $UCW$-hyperbolic space and $f$ be a proper, convex, lsc function on $X$. Let $x$, $y \in X$ be such that $d(x,\Prox_fx) \leq B$ and $d(x,y) \leq r$.

Then $d(y,\Prox_fy) \leq B + \sqrt{B^2 + 2B(B+r) + (r+B)^2}$.
\end{lemma}

\begin{proof}
Put $z:=\Prox_f x$, $w:=\Prox_f y$. From the previous lemma, we get that
$$f(z) - f(w) \leq d(x,z) \cdot d(w,z).$$
Put $C:=d(y,w)$. We have that
$$d(y,z) \leq d(y,x) + d(x,z) \leq r+B,$$
so
$$d(w,z) \leq d(w,y) + d(y,z) \leq C+r+B,$$
and
$$f(z) - f(w) \leq B(B+C+r).$$
We also have that
$$f(w) + \frac12d^2(y,w) \leq f(z) + \frac12d^2(y,z) \leq f(w) + B(B+C+r) + \frac12(r+B)^2,$$
so
$$\frac12C^2 \leq BC + B(B+r) + \frac12(r+B)^2,$$
from which we get that $C \leq B + \sqrt{B^2 + 2B(B+r) + (r+B)^2}$.
\end{proof}

\end{document}